\documentclass[12pt,a4paper]{article}%
\usepackage[utf8]{inputenc}
\usepackage{hyperref}
\usepackage{amsmath}
\usepackage{amsfonts}
\usepackage{amssymb}
\usepackage{xcolor}
\usepackage{graphicx}%
\usepackage{array}
\usepackage{lineno}
\providecommand{\U}[1]{\protect\rule{.1in}{.1in}}
\newtheorem{theorem}{Theorem}

\newtheorem{example}[theorem]{Example}

\newtheorem{lemma}[theorem]{Lemma}

\newtheorem{problem}[theorem]{Problem}
\newtheorem{proposition}[theorem]{Proposition}

\newenvironment{proof}[1][Proof]{\noindent\textbf{#1.} }{\ \hfill \rule{0.5em}{0.5em}\bigskip}
\graphicspath{{../Slike/}{Slike/}}

\begin{document}

\title{Spreads of degrees in graphs}
\author{Yair Caro$^{1}$, Riste \v{S}krekovski$^{2,3,4}$, Christina Zarb$^{5}$\\[0.3cm]
{\small $^{1}$ \textit{University of Haifa-Oranim, Kiryat Tiv'on, Israel}}\\[0.1cm]
{\small $^{2}$ \textit{University of Ljubljana, Faculty of Mathematics and Physics,
Ljubljana, Slovenia}}\\[0.1cm]
{\small $^{3}$ \textit{Faculty of Information Studies, Novo Mesto, Slovenia}}\\[0.1cm]
{\small $^{4}$ \textit{Rudolfovo -- Science and Technology Centre Novo Mesto, Slovenia}}\\[0.1cm]
{\small $^{5}$ \textit{Department of Mathematics, University of Malta, Malta}}}
\date{}
\maketitle

\begin{abstract}
For a graph $G$ and a set $B\subseteq V(G)$, the spread $\mathrm{sp}(B)$ of $B$ is the
difference between the largest and the smallest degree in $G$ of a vertex of $B$, and for an
integer $k\geq0$ the parameter $\mathrm{sp}(G,k)$ is the largest cardinality of a set $B$ with
$\mathrm{sp}(B)\leq k$. Caro, Lauri and Zarb derived a lower bound for $\mathrm{sp}(G,k)$ and, among several families of graphs, considered
\[
\mathrm{MOP}(n,k)=\min \{\mathrm{sp}(G,k):G\text{ is a maximal outerplanar graph of order
}n\}
\]
and determined $\mathrm{MOP}(n,k)$ up to an additive constant for every $k\not =2,$ leaving the
case $k=2$ open, with the bounds $4n/9\leq \mathrm{MOP}(n,2)\leq (5n+19)/11$. We first prove a lower bound on $\mathrm{sp}(G,k)$ for an arbitrary graph $G$ in
terms of its order $n$, its number of edges $m$ and its minimum degree $\delta $. This lower bound contains the bounds of Caro, Lauri and
Zarb and, for $k=0$, the bound $\mathrm{rep}(G)\geq \left\lceil n/(2d-2\delta +1)\right\rceil $ of Caro and West, where
$d=2m/n$.  We determine when
this lower bound is attained, exhibit explicit graphs attaining it, and show that it is exact for all graphs once
$n\geq n_{0}(\delta ,k,d)$. We then apply the bound to maximal outerplanar graphs: adjusting the count to this class we prove
\[
\mathrm{MOP}(n,2)\geq \left\lceil \frac{4n+10}{9}\right\rceil \qquad \text{for every }n\geq 14,
\]
with equality for $n\equiv 2\ (\mathrm{mod}\ 18)$, and $\mathrm{MOP}(n,2)=4n/9+O(1)$ for every $n$.
\end{abstract}



\section{Introduction}

Every graph on at least two vertices has two vertices of the same degree. This classroom observation
has been generalised in several directions, among them a characterisation of the graphs with only one repeated pair of
degrees \cite{BehzadChartrand}, the characterisation of graphic sequences
\cite{ErdosGallai, Hakimi, Havel}, and the study of the repetition number $\mathrm{rep}(G)$, the
maximum multiplicity of a value in the degree sequence of $G$, introduced by Caro and West
\cite{CaroWest}; for the degrees of maximal outerplanar graphs in particular see Jao and West
\cite{JaoWest}.

The generalisation we consider here is due to Erd\H{o}s, Chen, Rousseau and Schelp \cite{Erdos} and
was developed further by Caro, Lauri and Zarb \cite{CaroLauriZarb}. Let $G=(V,E)$ be a graph. For
$B\subseteq V$, the \emph{spread} of $B$ is
\[
\mathrm{sp}(B)=\max_{u\in B}\deg _{G}(u)-\min_{v\in B}\deg _{G}(v),
\]
and for an integer $k\geq 0$ we set
\[
\mathrm{sp}(G,k)=\max \{\left\vert B\right\vert :B\subseteq V(G)\text{ and }\mathrm{sp}
(B)\leq k\}.
\]
Thus $\mathrm{sp}(G,0)=\mathrm{rep}(G)$, and the first elementary result of graph theory reads
$\mathrm{sp}(G,0)\geq 2$ for every graph of order at least two. Erd\H{o}s, Chen, Rousseau and
Schelp proved that $\mathrm{sp}(G,k)\geq k+2$ whenever $G$ has at least $k+2$ vertices, and a short
proof avoiding the Erd\H{o}s--Gallai theorem is given in \cite{CaroLauriZarb}.

Notice that, writing $n_{j}$ for the number of vertices of degree $j$ in $G$, we have
\begin{equation}
\mathrm{sp}(G,k)=\max_{p\geq 0}\ (n_{p}+n_{p+1}+\cdots +n_{p+k}),  \label{For_window}%
\end{equation}

so that $\mathrm{sp}(G,k)$ depends on the degree sequence of $G$ only. We shall use
(\ref{For_window}) throughout, and we call the sets of $k+1$ consecutive degrees occurring in
(\ref{For_window}) \emph{windows}.

A graph is \emph{maximal outerplanar}, a MOP for short, if it is outerplanar and the
addition of any edge destroys outerplanarity; equivalently, a MOP of order $n\geq 3$ is a
triangulation of a convex $n$-gon. Notice that a MOP $G$ of order $n$ has $2n-3$ edges, so
$\sum_{v}\deg (v)=4n-6$, that $\delta (G)=2$, and that $G$ has $n-2$ triangles. The vertices of
degree $2$ are called \emph{ears}. For a MOP, the neighbourhood of every vertex $v$ induces a path,
the \emph{fan} at $v$, and consequently the number of triangles containing $v$ equals $\deg (v)-1$.

In \cite{CaroLauriZarb} the quantity
\[
\mathrm{MOP}(n,k)=\min \{\mathrm{sp}(G,k):G\text{ is a MOP of order }n\}
\]
was introduced, and the following bounds were established:
\[
\mathrm{MOP}(n,0)>\frac{n}{5},\qquad \mathrm{MOP}(n,1)\geq \frac{n}{3}+1,\qquad \mathrm{MOP}%
(n,2)\geq \frac{4n}{9},
\]
\[
\mathrm{MOP}(n,k)\geq \frac{(k-2)n}{k-1}\quad \text{for }k\geq 3.
\]
All of them, with the single exception of
$k=2$, are sharp up to a small additive constant. For $k=2$ the best construction known gives
$\mathrm{MOP}(n,2)\leq (5n+19)/11$ for $n\equiv 5\ (\mathrm{mod}\ 11)$, and the authors of
\cite{CaroLauriZarb} explicitly ask for the correct order of magnitude of $\mathrm{MOP}(n,2)$.

The paper has two parts. In Section \ref{Sec_general} we prove a lower bound on
$\mathrm{sp}(G,k)$ for an arbitrary graph $G$ in terms of its order $n$, its number of edges $m$ and
its minimum degree $\delta $ (Theorem \ref{Tm_gen}). It contains the bounds of \cite{CaroLauriZarb}
and, for $k=0$, the bound $\mathrm{rep}(G)\geq \left\lceil n/(2d-2\delta +1)\right\rceil $ of Caro and West
\cite{CaroWest}, where $d=2m/n$, and it refines them by an explicit nonnegative error term; we determine exactly when
it is attained and exhibit graphs attaining it for the corresponding parameters.   {This goal  is achieved  step by step, dividing Section 2 into three parts each having a target, and proving these targets gradually converge to this  goal.   The first target  in Section 2.1 is to obtain a lower-bound on $\mathrm{sp}(G,k)$ in terms of $n$, $m$ and $\delta$ and to optimize this bound regardless of whether the optimal sequences of non-negative integers used are graphical or not.  This is obtained in Theorem \ref{Tm_gen} and the lemmas following it.  The second target  in Section 2.2 aims to determine more precisely what these optimal sequences look like, still without the restriction of being graphic.  This is done via a series of technical lemmas culminating in Theorem \ref{Tm_sequence}.  The third target in Section 2.3  aims to show that the optimal sequences determined in Theorem \ref{Tm_sequence} are indeed  graphical, as finally proved by raising the sequence (Lemma \ref{Lemma_raise}) and applying the graphicality criterion of Lemma \ref{Lemma_ZZ}, and summarized in Theorem \ref{Tm_exact} providing $n \geq n_{0}(\delta ,k,d)$; explicit extremal graphs are given in Proposition \ref{Prop_construction}.}

The second part, Section \ref{Sec_mop}, determines $\mathrm{MOP}(n,2)$ up to an additive constant for every $n$,
and exactly for $n\equiv 2\ (\mathrm{mod}\ 18)$. The count of Theorem \ref{Tm_gen} adjusted to maximal outerplanar
graphs gives $\mathrm{MOP}(n,2)\geq \left\lceil (4n+6)/9\right\rceil $, and a weight using one further window
together with an ear-peeling argument raises this to
$\mathrm{MOP}(n,2)\geq \left\lceil (4n+10)/9\right\rceil $; we show that this bound has the right linear term.
Its equality analysis prescribes the
degree sequence which a maximal outerplanar graph attaining $4n/9$ must have; this degree sequence
cannot be realised exactly, but four exceptional vertices suffice, and we construct maximal outerplanar
graphs attaining the lower bound for $n\equiv 2\ (\mathrm{mod}\ 18)$ in Section \ref{Sec_construction};
for every $n$ the bound is attained up to an additive constant (Theorem \ref{Tm_main}); the remaining residue classes
can be handled by the same ladder with modified end caps, which we do not include here. In
particular the lower bound $4n/9$ of \cite{CaroLauriZarb} is asymptotically sharp.

The application of Theorem \ref{Tm_gen} to maximal planar graphs, Problem 2 of \cite{CaroLauriZarb}, is the
subject of the companion paper \cite{P2}.

\section{A lower bound for arbitrary graphs}

\label{Sec_general}

Our starting point is a lower bound on $\mathrm{sp}(G,k)$ for an arbitrary graph $G$ in terms of its
order, its size and its minimum degree. It is proved by a weighted count of windows of consecutive
degrees aligned at the minimum degree, it contains the bounds of \cite{CaroLauriZarb} and, for
$k=0$, the bound of Caro and West \cite{CaroWest} as special cases, and we determine exactly when it
is attained. The application in Section \ref{Sec_mop}, and the application to maximal planar graphs in the
companion paper \cite{P2}, adjust this count to the class of graphs at hand.

\subsection{The bound}

Throughout this section $G$ is an arbitrary graph of order $n$ with $m$ edges, minimum degree
$\delta =\delta (G)$, maximum degree $\Delta =\Delta (G)$ and average degree $d=d(G)=2m/n$. We fix
$k\geq 0$ and write $h=k+1$ for the width of a window. The \emph{aligned windows} are
\[
I_{i}=\{\delta +ih,\ \delta +ih+1,\ \ldots ,\ \delta +ih+k\},\qquad W_{i}=\sum_{j\in I_{i}}n_{j}%
\qquad (i\geq 0),
\]
so that $\sum_{i\geq 0}W_{i}=n$ and, by (\ref{For_window}), $W_{i}\leq \mathrm{sp}(G,k)$ for every
$i$.  {Let the indicator of $x$  with respect to $i$ be denoted by $\mu(x,i)$, where $\mu(x,i)=1$ if $x \in I_i$ and 0 otherwise.}  For an integer $t\geq 1$ put
\[
{u_{t}(x)=\sum_{i=0}^{t-1}(t-i)\cdot\mu(x,i),\qquad \varepsilon _{t}(x)=h\,u_{t}%
(x)-\left( \delta +th-x\right) .}
\]

{Observe that for $i \neq  j$, the intervals $I_i$ and $I_j$ are disjoint namely $|I_i \cap I_j | = 0$.}
\begin{lemma}
\label{Lemma_gen}For every integer $x\geq \delta $,
\[
\varepsilon _{t}(x)=\left\{
\begin{array}
[c]{ll}%
(x-\delta )\ \mathrm{mod}\ h, & \text{if }\delta \leq x<\delta +th,\\
x-\delta -th, & \text{if }x\geq \delta +th.
\end{array}
\right.
\]
In particular $\varepsilon _{t}(x)\geq 0$, that is $h\,u_{t}(x)\geq \delta +th-x$, with equality if
and only if $x=\delta +ih$ for some $0\leq i\leq t$.
\end{lemma}

\begin{proof}
Let $\delta \leq x<\delta +th$ and write $x=\delta +jh+r$ with $0\leq j\leq t-1$ and
$0\leq r\leq h-1$. Then $x\in I_{j}$ and $x\not \in I_{i}$ for $i\not =j$, so $u_{t}(x)=t-j$ and
$\varepsilon _{t}(x)=h(t-j)-\delta -th+x=x-\delta -jh=r$. If $x\geq \delta +th$, then $x$ lies in no
$I_{i}$ with $i\leq t-1$, so $u_{t}(x)=0$ and $\varepsilon _{t}(x)=x-\delta -th$. In both cases
$\varepsilon _{t}(x)\geq 0$, and it vanishes exactly when $r=0$, respectively when $x=\delta +th$.
\end{proof}

\begin{theorem}
\label{Tm_gen}Let $G$ be a graph of order $n$ with $m$ edges and minimum degree $\delta $, and let
$k\geq 0$ and $t\geq 1$. Then
\[
\mathrm{sp}(G,k)\ \geq \ \frac{2\left( n\left( \delta +t(k+1)\right) -2m+\sum_{v\in V(G)}%
\varepsilon _{t}(\deg (v))\right) }{(k+1)\,t(t+1)} .
\]
\end{theorem}

\begin{proof}
{By counting in two ways,} the definition of $u_{t}$ and by $W_{i}\leq \mathrm{sp}(G,k)$,
\[
\sum_{v}u_{t}(\deg (v))=\sum_{i=0}^{t-1}(t-i)W_{i}\leq \left( \sum_{i=0}^{t-1}(t-i)\right)
\mathrm{sp}(G,k)=\frac{t(t+1)}{2}\,\mathrm{sp}(G,k).
\]
On the other hand, the definition of $\varepsilon _{t}$ says precisely that
\[
h\,u_{t}(x)=\left( \delta +th-x\right) +\varepsilon _{t}(x)\qquad \text{for every integer }x\geq \delta .
\]
Applying this to $x=\deg (v)$ for each vertex $v$ and adding the $n$ identities,
\begin{align*}
h\sum_{v}u_{t}(\deg (v))\  &=\ \sum_{v}\Bigl[ \bigl( \delta +th-\deg (v)\bigr) +\varepsilon _{t}(\deg (v))\Bigr] \\
&=\ \sum_{v}\left( \delta +th\right) \ -\ \sum_{v}\deg (v)\ +\ \sum_{v}\varepsilon _{t}(\deg (v))\\
&=\ n(\delta +th)-2m+\sum_{v}\varepsilon _{t}(\deg (v)),
\end{align*}
the last line because the first sum has $n$ equal terms and $\sum_{v}\deg (v)=2m$. Recall that $h=k+1$;
comparing the two displays and dividing by $h\,t(t+1)/2$ gives the asserted bound.
\end{proof}

Since $\varepsilon _{t}\geq 0$, Theorem \ref{Tm_gen} contains the bound
\begin{equation}
\mathrm{sp}(G,k)\ \geq \ \frac{2n\left( \delta +t(k+1)-d\right) }{(k+1)\,t(t+1)},
\label{For_clz4}%
\end{equation}
which is precisely the bound (4) of \cite{CaroLauriZarb}. Thus Theorem \ref{Tm_gen} refines every
bound derived there from {(4) of \cite{CaroLauriZarb}}, the gain being the explicit term
$\sum_{v}\varepsilon _{t}(\deg (v))$, which measures how far the degrees of $G$ are from the
arithmetic progression $\delta ,\delta +(k+1),\ldots ,\delta +t(k+1)$. We write
\[
f_{t}=f_{t}(n,m,\delta ,k)=\frac{2\left( n(\delta +t(k+1))-2m\right) }{(k+1)\,t(t+1)}=\frac{2n(t-x)}%
{t(t+1)},\qquad x=\frac{d-\delta }{k+1}\geq 0,
\]
for the right-hand side of (\ref{For_clz4}). Notice that in \cite{CaroLauriZarb} the integer $t$ is
tied to $\mathrm{sp}(G,k)$ by $n=t\,\mathrm{sp}(G,k)+b$, whereas Theorem \ref{Tm_gen} holds for
every $t\geq 1$, so that one may take the maximum over $t$. The best $t$ is easily located.

\begin{lemma}
\label{Lemma_bestt}For real $x\geq 0$ and integer $t\geq 1$ we have $f_{t}\geq f_{t+1}$ if and only
if $t\geq 2x$, with equality if and only if $t=2x$. Hence $\max_{t\geq 1}f_{t}$ is attained at
$t=\max \{1,\left\lceil 2x\right\rceil \}$, also at $t=2x+1$ when $2x$ is a positive integer, and at no
other $t$.
\end{lemma}

\begin{proof}
Multiplying by $t(t+1)(t+2)/(2n)>0$, the inequality $f_{t}\geq f_{t+1}$ is equivalent to
$(t-x)(t+2)\geq (t+1-x)t$, that is to $t^{2}+2t-xt-2x\geq t^{2}+t-xt$, that is to $t\geq 2x$; the
two sides are equal exactly when $t=2x$. Consequently the sequence $(f_{t})_{t\geq 1}$ is strictly
increasing while $t<2x$, satisfies {$f_{t}=f_{t+1}=n/(2x+1)$} for $t=2x$, and is strictly decreasing once
$t>2x$. {However $t$ must be a positive integer hence we write} $T=\max \{1,\left\lceil 2x\right\rceil \}$ and distinguish three cases.

If $2x<1$, then $T=1$ and every $t\geq 1$ satisfies $t>2x$, so $f_{1}>f_{2}>\cdots $ and the
maximum is attained at $t=1=T$ only.

If $2x\geq 1$ is not an integer, then $T=\left\lceil 2x\right\rceil \geq 2$ and $T-1<2x<T$. For
$1\leq t\leq T-1$ we have $t<2x$ and hence $f_{t}<f_{t+1}$, and for $t\geq T$ we have $t>2x$ and
hence $f_{t}>f_{t+1}$. Thus $f_{1}<\cdots <f_{T}>f_{T+1}>\cdots $, and the maximum is attained at
$t=T$ only.

If $2x\geq 1$ is an integer, then $T=2x$. As before $f_{t}<f_{t+1}$ for $t<2x$ and
$f_{t}>f_{t+1}$ for $t>2x$, while now $t=2x$ is the equality case, so that
\[
f_{1}<f_{2}<\cdots <f_{2x}=f_{2x+1}>f_{2x+2}>\cdots .
\]
Hence the maximum is attained exactly at the two values $t=2x$ and $t=2x+1$, its value being
\[
f_{t}=f_{2x}=\frac{2nx}{2x(2x+1)}=\frac{n}{2x+1}=\frac{2n(x+1)}{(2x+1)(2x+2)}=f_{t+1}.
\] The three cases exhaust the possibilities, and the maximum is
attained at two values of $t$ exactly in the third of them, that is exactly when $2x$ is a positive
integer.
\end{proof}

\begin{proposition}
\label{Prop_compare}Let $x=(d-\delta )/(k+1)$.

\begin{itemize}
\item[\textrm{(i)}] If $2x$ is an integer, then $f_{2x+1}=n/(2x+1)$, and also $f_{2x}=n/(2x+1)$
when $2x\geq 1$.

\item[\textrm{(ii)}] For every real $x\geq 0$,
\[
\max_{t\geq 1}f_{t}\ \geq \ \frac{n}{2x+1}=\frac{n(k+1)}{2d-2\delta +k+1},
\]
with equality if and only if $2x$ is an integer.

\item[\textrm{(iii)}] $\mathrm{sp}(G,k)\geq n(k+1)/(2\Delta -2d+k+1)$.
\end{itemize}

\end{proposition}

\begin{proof}
(i) Both assertions are a single computation. For $t=2x+1$, which is a positive integer for every
$x\geq 0$ with $2x$ an integer,
\[
f_{2x+1}=\frac{2n\bigl( (2x+1)-x\bigr) }{(2x+1)(2x+2)}=\frac{2n(x+1)}{2(2x+1)(x+1)}=\frac{n}{2x+1};
\]
and if moreover $2x\geq 1$, so that $t=2x$ is a positive integer as well,
\[
f_{2x}=\frac{2n(2x-x)}{2x(2x+1)}=\frac{n}{2x+1} .
\]
The restriction $2x\geq 1$ in the second case is needed only because $t$ must be at least one.

(ii) For $x=0$ we have $f_{1}=n$. For $x>0$ put $t=\left\lceil 2x\right\rceil $ and $\sigma =2x$, so
that $t-1<\sigma \leq t$. The inequality $f_{t}\geq n/(2x+1)$ reads $g(\sigma ):=(\sigma +1)(2t-\sigma
)\geq t(t+1)$; the quadratic $g$ is concave with $g(t-1)=g(t)=t(t+1)$, so $g(\sigma )\geq t(t+1)$ on
$[t-1,t]$, with equality only at the end points. Hence $f_{t}\geq n/(2x+1)$, strictly when
$t-1<\sigma <t$, that is when $2x$ is not an integer. If $2x$ is an integer, then by Lemma
\ref{Lemma_bestt} the maximum is $f_{2x}=f_{2x+1}=n/(2x+1)$ by (i).

(iii) Recall from \cite{CaroLauriZarb} that $\mathrm{sp}(G,k)=\mathrm{sp}(\overline{G},k)$, because
$\deg _{G}(u)-\deg _{G}(v)=\deg _{\overline{G}}(v)-\deg _{\overline{G}}(u)$. The complement has
minimum degree $n-1-\Delta $ and average degree $n-1-d$, so (ii) applied to $\overline{G}$ gives the
bound.
\end{proof}

\subsection{Equality, and the bound as a statement about degree sequences}

The proof of Theorem \ref{Tm_gen} uses only the degree sequence of $G$. We therefore call a
nondecreasing sequence $D=(d_{1}\leq \cdots \leq d_{n})$ of integers with $d_{1}=\delta $ a
\emph{$\delta $-sequence}, and we call the largest number of its terms lying in an interval of $k+1$
consecutive integers its \emph{window number} $\mathrm{sp}(D,k)$; thus $\mathrm{sp}(G,k)=\mathrm{sp}%
(D,k)$ for the degree sequence $D$ of $G$, and Theorem \ref{Tm_gen} holds, with the same proof, for
every $\delta $-sequence with sum $2m$.  {The next proposition translates Theorem \ref{Tm_gen} into a structure for the extremal case when equality in the lower bound is attained.}

\begin{proposition}
\label{Prop_equality}Let $G$ be a graph of order $n$ with $m$ edges and minimum degree $\delta $, let
$k\geq 0$, $t\geq 1$, $h=k+1$ and $s=\mathrm{sp}(G,k)$.

\begin{itemize}
\item[\textrm{(i)}] Equality holds in Theorem \ref{Tm_gen} if and only if $W_{0}=W_{1}=\cdots
=W_{t-1}=s$.

\item[\textrm{(ii)}] Equality holds in (\ref{For_clz4}), that is $s=f_{t}$, if and only if every
degree of $G$ lies in $\{\delta ,\delta +h,\ldots ,\delta +th\}$ and
\[
n_{\delta }=n_{\delta +h}=\cdots =n_{\delta +(t-1)h}=s,\qquad n_{\delta +th}=n-ts\leq s.
\]
In particular equality forces $s=f_{t}$ to be an integer with $n/(t+1)\leq s\leq n/t$, and
$t-1\leq 2x\leq t$, {compatible with Lemma \ref{Lemma_bestt}.}
\end{itemize}
\end{proposition}

\begin{proof}
(i) The only inequality in the proof of Theorem \ref{Tm_gen} is $\sum_{i<t}(t-i)W_{i}\leq
\sum_{i<t}(t-i)s$, and its coefficients are positive.

(ii) The bound (\ref{For_clz4}) {of Theorem \ref{Tm_gen}} is obtained from (i) by dropping $\sum_{v}\varepsilon _{t}(\deg
(v))\geq 0$, and by Lemma \ref{Lemma_gen} this sum vanishes exactly when every degree is of the form
$\delta +ih$ with $0\leq i\leq t$. In that case $W_{i}=n_{\delta +ih}$, so (i) gives
$n_{\delta +ih}=s$ for $i<t$, hence $n_{\delta +th}=n-ts$, and the window
$[\delta +th,\delta +th+k]$ contains only the degree $\delta +th$, so $n-ts\leq s$. Conversely these
conditions give $\varepsilon _{t}\equiv 0$ and $W_{i}=s$. The degree sum of such a graph is
$\sum_{i<t}s(\delta +ih)+(n-ts)(\delta +th)=n(\delta +th)-sht(t+1)/2$, which is $2m$ exactly when
$s=f_{t}$; the inequalities $0\leq n-ts\leq s$ read $n/(t+1)\leq s\leq n/t$, and substituting them
into $d=2m/n$ gives $\delta +(t-1)h/2\leq d\leq \delta +th/2$, that is $t-1\leq 2x\leq t$.
\end{proof}

For $1\leq s\leq n$ put $t=\left\lfloor n/s\right\rfloor $ and $b=n-ts\in \lbrack 0,s)$, and let
the \emph{packed} $\delta $-sequence $P(n,s)$ consist of $s$ copies of each of $\delta ,\delta
+h,\ldots ,\delta +(t-1)h$ and $b$ copies of $\delta +th$; its sum is
\[
\Sigma (n,s)=n\delta +h\sum_{i=1}^{n}\left\lfloor \frac{i-1}{s}\right\rfloor =n\delta +h\left( \frac{%
st(t-1)}{2}+bt\right) .
\]
For $h=1$ these are exactly the \emph{packed lists} of Caro and West \cite{CaroWest}.

\begin{lemma}
\label{Lemma_sorted}Let $D$ be a $\delta $-sequence and $s\geq 1$. Then $\mathrm{sp}(D,k)\leq s$ if
and only if $d_{i+s}\geq d_{i}+h$ for all $1\leq i\leq n-s$. Consequently every $\delta $-sequence
with $\mathrm{sp}(D,k)\leq s$ satisfies $d_{i}\geq \delta +h\left\lfloor (i-1)/s\right\rfloor $ for
every $i$ and has sum at least $\Sigma (n,s)$, and $P(n,s)$ is the unique such sequence with sum
exactly $\Sigma (n,s)$; moreover $\mathrm{sp}(P(n,s),k)=s$.
\end{lemma}

\begin{proof}
If $d_{i+s}\leq d_{i}+h-1$, the $s+1$ terms $d_{i},\ldots ,d_{i+s}$ lie in the window $[d_{i},d_{i}+h-1]$.
Conversely, if a window contains $s+1$ terms, they are consecutive in the sorted order, say
$d_{i},\ldots ,d_{i+s}$, and $d_{i+s}-d_{i}\leq h-1$. The inequality
$d_{i}\geq \delta +h\left\lfloor (i-1)/s\right\rfloor $ follows from $d_{1}=\delta $ and
$d_{i+s}\geq d_{i}+h$ by induction, and it is an equality for every $i$ exactly for $P(n,s)$. Finally,
the values of $P(n,s)$ differ by multiples of $h$, so a window of $h$ consecutive integers contains at
most one of them, with multiplicity at most $s$.
\end{proof}

\begin{lemma}
\label{Lemma_packed}Let $1\leq s\leq n$, $t=\left\lfloor n/s\right\rfloor $ and $2m=\Sigma (n,s)$. Then
$f_{t}(n,m,\delta ,k)=s$ and $\max_{t^{\prime }\geq 1}f_{t^{\prime }}(n,m,\delta ,k)=s$.
\end{lemma}

\begin{proof}
With $b=n-ts$,
\[
f_{t}=\frac{2\left( n(\delta +th)-\Sigma (n,s)\right) }{ht(t+1)}=\frac{2\left( nt-st(t-1)/2-bt\right) %
}{t(t+1)}=\frac{2n-s(t-1)-2b}{t+1},
\]
and since $2n-2b=2ts$ the last fraction equals $s(t+1)/(t+1)=s$. Applying Theorem \ref{Tm_gen} to the $\delta $-sequence $P(n,s)$, whose window number is $s$ by
Lemma \ref{Lemma_sorted}, gives $f_{t^{\prime }}\leq s$ for every $t^{\prime }$.
\end{proof}

For $2m\geq n\delta $ let
\[
s^{\ast }(n,m,\delta ,k)=\min \{s\geq 1:\Sigma (n,s)\leq 2m\},
\]
which is well defined because $\Sigma (n,n)=n\delta $.

\begin{theorem}
\label{Tm_sequence}Let $n\geq 2$, $\delta \geq 0$, $k\geq 0$ and $2m\geq n\delta $. Then
\[
\max_{t\geq 1}\left\lceil f_{t}(n,m,\delta ,k)\right\rceil \ =\ s^{\ast }(n,m,\delta ,k)\ =\ \min
\{\mathrm{sp}(D,k):D\in \mathcal{D}\},
\]
where $\mathcal{D}$ is the set of all $\delta $-sequences of length $n$ with sum $2m$.
\end{theorem}

\begin{proof}
Write $s^{\ast }=s^{\ast }(n,m,\delta ,k)$. A $\delta $-sequence with window number $s<s^{\ast }$
has sum at least $\Sigma (n,s)>2m$ by Lemma \ref{Lemma_sorted}, so the minimum on the right is at
least $s^{\ast }$. Conversely $P(n,s^{\ast })$ has sum $\Sigma (n,s^{\ast })\leq 2m$ and window number
$s^{\ast }$, and raising its largest term by one, as often as needed, produces a $\delta $-sequence
with sum $2m$ and window number at most $s^{\ast }$: raising the last term of a nondecreasing sequence
preserves the inequalities $d_{i+s}\geq d_{i}+h$ of Lemma \ref{Lemma_sorted}, and the minimum stays
$\delta $ because $n\geq 2$. So the minimum equals
$s^{\ast }$.

Theorem \ref{Tm_gen} applied to that sequence gives $f_{t}\leq s^{\ast }$ for every $t$, hence
$\max_{t}\left\lceil f_{t}\right\rceil \leq s^{\ast }$. For the reverse inequality, if $s^{\ast }=1$
then $f_{t}>0$ for $t>x$. If $s^{\ast }\geq 2$, put $s=s^{\ast }-1$ and $t=\left\lfloor
n/s\right\rfloor $; then $\Sigma (n,s)>2m$ by the minimality of $s^{\ast }$, and Lemma
\ref{Lemma_packed} gives
\[
f_{t}(n,m,\delta ,k)=\frac{2\left( n(\delta +th)-2m\right) }{ht(t+1)}>\frac{2\left( n(\delta
+th)-\Sigma (n,s)\right) }{ht(t+1)}=s=s^{\ast }-1,
\]
so $\left\lceil f_{t}\right\rceil \geq s^{\ast }$.
\end{proof}

 {We observe that, calling
$s^{\ast }(n,m,\delta ,k)$ the \emph{maximised bound}, we see that the maximised bound is attained by a graph of order
$n$ with $m$ edges and minimum degree $\delta $ if and only if some $\delta $-sequence with sum
	$2m$ and window number $s^{\ast }$ is the degree sequence of a graph.}

Theorem \ref{Tm_sequence} says that the only loss in Theorem \ref{Tm_gen} is graphicality: the
maximised bound is the exact minimum of the window number over all integer sequences with the given
$n$, $m$ and $\delta $. For $k=0$ this is the content of the \emph{packed lists} of
\cite[Section 2]{CaroWest}, and it also identifies $\max_{t}\left\lceil f_{t}\right\rceil $ with the
refined count in the proof of \cite[Lemma 2.1]{CaroWest}, which keeps a remainder term that the bound
$n/(2d-2\delta +1)$ drops; by Proposition \ref{Prop_compare}(ii) the two agree exactly when $2d$ is
an integer.

\subsection{Realisability}

We now show that the bound is attained whenever its equality sequence is graphic, that this is the
case for all large $n$, and that for large $n$ the maximised bound is exact. We use a classical
sufficient condition for graphicality, due to Zverovich and Zverovich
\cite{ZverovichZverovich}; we include its short proof for completeness.

\begin{lemma}
\label{Lemma_ZZ}Let $D=(d_{1}\geq \cdots \geq d_{n})$ be a sequence of integers with even sum,
largest term $\Delta $ and smallest term $\delta \geq 1$. If $n\geq (\Delta +\delta +1)^{2}/(4\delta
)$, then $D$ is the degree sequence of a graph.
\end{lemma}

\begin{proof}
We verify the Erd\H{o}s--Gallai conditions \cite{ErdosGallai}
$\sum_{i\leq j}d_{i}\leq j(j-1)+\sum_{i>j}\min \{j,d_{i}\}$ for $1\leq j\leq n$. Notice first that
$(\Delta +\delta +1)^{2}-4\delta (\Delta +1)=(\Delta +1-\delta )^{2}\geq 0$, so the hypothesis gives
$n\geq \Delta +1$. If $j\geq \Delta +1$, the right-hand side is at least $j(j-1)\geq j\Delta $, which
is at least the left-hand side. If $j\leq \delta $, then $\min \{j,d_{i}\}=j$ and the right-hand side
is $j(j-1)+(n-j)j=j(n-1)\geq j\Delta $. If $\delta <j\leq \Delta $, then $\min \{j,d_{i}\}\geq \delta
$, and it suffices that $j\Delta \leq j(j-1)+(n-j)\delta $, that is $n\delta \geq j(\Delta +\delta
+1-j)$; the right-hand side is at most $(\Delta +\delta +1)^{2}/4$.
\end{proof}

\begin{theorem}
\label{Tm_realisable}Let $\delta \geq 1$, $k\geq 0$, $h=k+1$ and $t\geq 1$, and let $n\geq 2$ and $m$ be integers
with $2m\geq n\delta $. There is a graph $G$ of order $n$ with $m$ edges and minimum degree $\delta $ such that $\mathrm{sp}(G,k)=f_{t}(n,m,\delta ,k)$
if and only if there is an integer $s$ with
\[
\frac{n}{t+1}\leq s\leq \frac{n}{t}\qquad \text{and}\qquad 2m=n(\delta +th)-\frac{sht(t+1)}{2}
\]
such that the sequence $P$ consisting of $s$ copies of each of $\delta ,\delta +h,\ldots ,\delta
+(t-1)h$ and $n-ts$ copies of $\delta +th$ is graphic. This is the case whenever
\[
n\geq \frac{(2\delta +th+1)^{2}}{4\delta },
\]
and for $k=0$ it is the case if and only if $n$ exceeds the largest term of $P$
\cite[Theorem 2.4]{CaroWest}.
\end{theorem}

\begin{proof}
By Proposition \ref{Prop_equality}(ii) a graph attaining $f_{t}$ has degree sequence $P$ with $s$ as
displayed, and conversely a graph with degree sequence $P$ has $\mathrm{sp}(G,k)=s=f_{t}$, because
the windows $[\delta +ih,\delta +ih+k]$ contain $s$, respectively $n-ts\leq s$, of its vertices and
no other window contains more. The largest term of $P$ is at most $\delta +th$, so Lemma
\ref{Lemma_ZZ} applies as soon as $n\geq (2\delta +th+1)^{2}/(4\delta )$. For $k=0$ the sequence $P$
is a packed list in the sense of \cite{CaroWest}, and \cite[Theorem 2.4]{CaroWest} states that a
packed list with even sum is graphic if and only if its length exceeds its largest term.
\end{proof}

Thus, for $n\geq (2\delta +th+1)^{2}/(4\delta )$, equality holds in (\ref{For_clz4}) {of Theorem \ref{Tm_gen}} for every $s$
with $n/(t+1)\leq s\leq n/t$ for which the display makes $m$ an integer, and the corresponding
average degrees $d$ run through the range $t-1\leq 2x\leq t$ in steps of $ht(t+1)/(2n)$. For $k=0$ and
$t=2x+1$ this is the statement of \cite{CaroWest} that equality can hold when $2d$ is an integer, and
the condition $n\geq 2d-\delta +1$ of \cite[Example 2.3]{CaroWest} is recovered exactly. {We now exhibit a construction for extremal graphs.}


\begin{proposition}
\label{Prop_construction}In the notation of Theorem \ref{Tm_realisable} let $b=n-ts$, and suppose
that $s\geq \delta +(t-1)h+1$, that $s(\delta +ih)$ is even for $0\leq i<t$, and that either $b=0$
or $b\geq \delta +th+1$ with $b(\delta +th)$ even. Let $G_{0}$ be the disjoint union of an
$(\delta +ih)$-regular circulant graph on $s$ vertices for each $0\leq i<t$ and, if $b>0$, of a
$(\delta +th)$-regular circulant graph on $b$ vertices. Then $G_{0}$ has degree sequence $P$ and
$\mathrm{sp}(G_{0},k)=s=f_{t}$; and for $\delta \geq 2$ there is also a connected graph with the same
degree sequence.
\end{proposition}

\begin{proof}
Let $r\geq q+1$ with $qr$ even. The circulant graph on $0,1,\ldots ,r-1$ in which $i$ is joined to
$i\pm 1,\ldots ,i\pm \left\lfloor q/2\right\rfloor $, and also to $i+r/2$ when $q$ is odd, is
$q$-regular and, for $q\geq 2$, contains the Hamiltonian cycle $0,1,\ldots ,r-1$. The hypotheses
make every class admissible, so $G_{0}$ has degree sequence $P$, and $\mathrm{sp}(G_{0},k)=s$ as in
the proof of Theorem \ref{Tm_realisable}. For $\delta \geq 2$ every component is a circulant graph of
degree at least $2$ containing the connection $1$, and deleting one edge from it leaves it connected,
since the Hamiltonian cycle survives, or becomes a Hamiltonian path. Merge the components one at a
time: having merged some of them into a connected graph $U$, choose an edge $ab$ of $U$ and an edge $cd$
of a fresh component $C$, delete $ab$ and $cd$ and add $ac$ and $bd$. Degrees are preserved, $C-cd$ is
connected, and the new edges join the component of $a$ and the component of $b$ in $U-ab$ to $C-cd$,
so the result is connected. After all components have been merged we have a connected graph with degree
sequence $P$.
\end{proof}

\noindent{\textbf{Remark}:  the choice of Cayley graphs in the above construction is for convenience.   There are other regular graphs that can serve in the same role coming from 1-factorization and 2-factorizatons of $K_n$  by using a suitable number of 1-factors and 2-factors \cite{harary1969}. }

\begin{lemma}
\label{Lemma_raise}Let $2m\geq n\delta $, $n\geq 2$, $s=s^{\ast }(n,m,\delta ,k)\geq 2$,
$t=\left\lfloor n/s\right\rfloor $, $E=2m-\Sigma (n,s)$ and $E=q(n-1)+\rho $ with $0\leq \rho <n-1$.
Let $R$ be obtained from $P(n,s)=(p_{1}\leq \cdots \leq p_{n})$ by adding $q$ to $p_{2},\ldots ,p_{n}$
and one more to the last $\rho $ terms. Then:

\begin{itemize}
\item[\textrm{(i)}] $R$ is a $\delta $-sequence with sum $2m$, window number $s$ and largest term at
most $\delta +th+q+1$;

\item[\textrm{(ii)}] $\Sigma (n,s-1)-\Sigma (n,s)<h\left( \frac{n(n-1)}{2s(s-1)}+n\right) $, and
consequently $q<h\left( \frac{(t+1)^{2}}{n}+2\right) $;

\item[\textrm{(iii)}] $t<2x+2$.
\end{itemize}
\end{lemma}

\begin{proof}
(i) The sum and the first term are clear. The amounts added are nondecreasing in $i$ for $i\geq 2$,
so $r_{i+s}-r_{i}\geq p_{i+s}-p_{i}\geq h$ for $i\geq 2$, and also for $i=1$ because
$r_{1+s}\geq p_{1+s}\geq p_{1}+h=r_{1}+h$; by Lemma \ref{Lemma_sorted} the window number is at most
$s$, and it is at least $s^{\ast }=s$ by Theorem \ref{Tm_sequence}. The largest term is
$p_{n}+q+1\leq \delta +h\left\lfloor (n-1)/s\right\rfloor +q+1\leq \delta +th+q+1$.

(ii) $\Sigma (n,s-1)-\Sigma (n,s)=h\sum_{y=0}^{n-1}\left( \left\lfloor y/(s-1)\right\rfloor
-\left\lfloor y/s\right\rfloor \right) $, and each term is less than $\frac{y}{s-1}-\frac{y}{s}+1=%
\frac{y}{s(s-1)}+1$. By the minimality of $s^{\ast }$ we have $E<\Sigma (n,s-1)-\Sigma (n,s)$, so
$q\leq E/(n-1)<h\left( \frac{n}{2s(s-1)}+\frac{n}{n-1}\right) $. Now $n<(t+1)s$ gives
$\frac{n}{2s(s-1)}<\frac{t+1}{2(s-1)}\leq \frac{t+1}{s}<\frac{(t+1)^{2}}{n}$, and $\frac{n}{n-1}\leq 2$.

(iii) For $t=1$ there is nothing to prove. For $t\geq 2$, $\Sigma (n,s)\leq 2m$ gives
$nhx=2m-n\delta \geq h\frac{st(t-1)}{2}>h\frac{n}{t+1}\cdot \frac{t(t-1)}{2}$, so
$2x>\frac{t(t-1)}{t+1}=t-2+\frac{2}{t+1}>t-2$.
\end{proof}

{We finish this section with Theorem \ref{Tm_exact} proving that the conditions for optimality developed in steps in  Theorem \ref{Tm_gen} , Theorem  \ref{Tm_sequence}  and Theorem \ref{Tm_realisable}  for integer sequences are realized by graphical sequences  provided that  $n$ is large.} Throughout we write
\[
n_{0}(\delta ,k,d)=\max \left\{ \frac{(2d+5h+1)^{2}}{4\delta },\ \left( \frac{2(d-\delta )}{h}+3\right)
^{2}\right\}
\]
for the threshold appearing in that theorem, so that the maximised bound is exact for every
$n\geq n_{0}(\delta ,k,d)$.
\begin{theorem}
\label{Tm_exact}Let $\delta \geq 1$, $k\geq 0$, $h=k+1$, and let $n$ and $m$ be integers with
$d=2m/n\geq \delta $, $x=(d-\delta )/h$ and
\[
n\ \geq \ \max \left\{ \frac{(2d+5h+1)^{2}}{4\delta },\ (2x+3)^{2}\right\} .
\]
Then
\[
\min \{\mathrm{sp}(G,k):G\text{ of order }n\text{ with }m\text{ edges and }\delta (G)=\delta \}\ =\
\max_{t\geq 1}\left\lceil f_{t}(n,m,\delta ,k)\right\rceil .
\]
\end{theorem}

\begin{proof}
The minimum is at least the maximum by Theorem \ref{Tm_gen}. Let $s=s^{\ast }(n,m,\delta ,k)$. If
$s=1$, then $\Sigma (n,1)=n\delta +hn(n-1)/2\leq 2m$, that is $n\leq 2x+1<(2x+3)^{2}\leq n$, which is
absurd; so $s\geq 2$, and Lemma \ref{Lemma_raise} applies. Let $t=\left\lfloor n/s\right\rfloor $ and
let $R$ be the raised sequence. By (iii), $t+1<2x+3$, so $n\geq (2x+3)^{2}>(t+1)^{2}$, and (ii)
gives $q<3h$, that is $q\leq 3h-1$. Hence by (i) the largest term of $R$ is at most
$\delta +th+3h\leq \delta +(t+3)h<\delta +(2x+5)h=2d-\delta +5h$. The sequence $R$ has even sum
$2m$, smallest term $\delta \geq 1$ and largest term less than $2d-\delta +5h$, so by Lemma
\ref{Lemma_ZZ} it is graphic as soon as $n\geq (2d+5h+1)^{2}/(4\delta )$. A graph $G$ with degree
sequence $R$ has minimum degree $\delta $, $m$ edges and $\mathrm{sp}(G,k)=s=\max_{t}\left\lceil
f_{t}\right\rceil $ by Theorem \ref{Tm_sequence}.
\end{proof}


\begin{example}
\label{Ex_exact}Let $\delta =3$, $k=4$ and $d=6$, so that $h=5$ and $x=(d-\delta )/h=3/5$. Theorem
\ref{Tm_exact} applies as soon as
\[
n\ \geq \ \max \left\{ \frac{(2d+5h+1)^{2}}{4\delta },\ (2x+3)^{2}\right\} =\max \left\{ \frac{361}{3},\
\frac{441}{25}\right\} ,\qquad \text{that is}\qquad n\geq 121 .
\]
By Lemma \ref{Lemma_bestt} the best $t$ is $\max \{1,\lceil 2x\rceil \}=2$, and
\[
f_{1}=\frac{2n}{5},\qquad f_{2}=\frac{2n(2-3/5)}{6}=\frac{7n}{15},\qquad f_{3}=\frac{2n}{5},
\]
so the maximised bound is $\left\lceil 7n/15\right\rceil $. Take $n=210$ and $m=630$. The bound is
$7\cdot 210/15=98$, and by Proposition \ref{Prop_equality}(ii) a graph attaining it must have the degree
sequence
\[
3^{(98)}\,8^{(98)}\,13^{(14)},
\]
whose sum is $294+784+182=1260=2m$. Here $s=98$, $t=2$ and $b=n-ts=14$, so
$s\geq \delta +(t-1)h+1=9$, the products $98\cdot 3$ and $98\cdot 8$ are even, and $b=14\geq \delta +th+1=14$
with $14\cdot 13$ even: Proposition \ref{Prop_construction} applies. An optimal graph is therefore the
disjoint union of the $3$-regular circulant on $98$ vertices (join $i$ to $i\pm 1$ and to $i+49$), the
$8$-regular circulant on $98$ vertices (join $i$ to $i\pm 1,\ldots ,i\pm 4$) and the complete graph $K_{14}$,
which is the $13$-regular circulant on $14$ vertices; and by the same proposition a connected graph with this
degree sequence exists as well. Its three degrees are pairwise at distance at least $5=k+1$, so every window
of five consecutive integers contains just one of them and
$\mathrm{sp}(G,4)=\max \{98,98,14\}=98$, which is the bound. Note how far the threshold is from being
necessary: the equality sequence above is graphic already for $n=15$, and Proposition
\ref{Prop_construction} realises it for every $n\equiv 0\ (\mathrm{mod}\ 30)$ with $n\geq 210$.
\end{example}

\section{Application: maximal outerplanar graphs}
\label{Sec_mop}

\subsection{The lower bound}

Throughout, $G$ is a MOP of order $n$ and $n_{j}$ denotes the number of vertices of degree $j$ in
$G$. For $p\geq 2$ we abbreviate
\[
W_{[p,p+2]}=n_{p}+n_{p+1}+n_{p+2},
\]
so that by (\ref{For_window}) every one of these \emph{windows} satisfies
$W_{[p,p+2]}\leq \mathrm{sp}(G,2)$. 

The simplest weighted count already explains the shape of the answer: Lemma \ref{Lemma_u} shows that a maximal
outerplanar graph attaining the bound has to be built on vertices of degrees $2$, $5$ and $8$.
\begin{lemma} 
\label{Lemma_u}Let $u(d)=2$ for $2\leq d\leq 4$, $u(d)=1$ for $5\leq d\leq 7$ and $u(d)=0$ for $d\geq 8$, and put
$\varepsilon (d)=3u(d)-(8-d)$. Then $\varepsilon (d)\geq 0$ for every $d\geq 2$, with equality if and only if
$d\in \{2,5,8\}$, and every MOP $G$ of order $n$ satisfies
\[
9\,\mathrm{sp}(G,2)\ \geq \ 4n+6+\sum_{v\in V(G)}\varepsilon (\deg (v)) .
\]
In particular $\mathrm{MOP}(n,2)\geq \left\lceil (4n+6)/9\right\rceil $.
\end{lemma}

\begin{proof}
For $2\leq d\leq 4$ we get $\varepsilon (d)=d-2\in \{0,1,2\}$, for
$5\leq d\leq 7$ we get $\varepsilon (d)=d-5\in \{0,1,2\}$, and for $d\geq 8$ we get
$\varepsilon (d)=d-8\geq 0$; so $\varepsilon (d)\geq 0$ for every $d\geq 2$, that is
$3u(d)\geq 8-d$, with equality if and only if $d\in \{2,5,8\}$, which is the first assertion. The identity
$2W_{[2,4]}+W_{[5,7]}=\sum_{v}u(\deg (v))$ is immediate from the definition of $u$, and
$W_{[2,4]},W_{[5,7]}\leq \mathrm{sp}(G,2)$, so that with $\sum_{v}\deg (v)=4n-6$,
\begin{align*}
9\,\mathrm{sp}(G,2)\ &\geq \ 3\left( 2W_{[2,4]}+W_{[5,7]}\right) \ =\ \sum_{v}3u(\deg (v))\\
&=\ \sum_{v}\left( 8-\deg (v)\right) +\sum_{v}\varepsilon (\deg (v))\ =\ 4n+6+\sum_{v}\varepsilon (\deg (v)).
\end{align*}
Since $\varepsilon \geq 0$ and $\mathrm{sp}(G,2)$ is an integer, this also gives
$\mathrm{MOP}(n,2)\geq \left\lceil (4n+6)/9\right\rceil $.
\end{proof}

Lemma \ref{Lemma_u} tells us what a construction attaining $4n/9$ must look like. If a sequence of MOPs satisfies
$\mathrm{sp}(G,2)=\frac49n+o(n)$, then $\sum_{v}\varepsilon (\deg (v))=o(n)$ by the lemma, and since
$\varepsilon (d)\geq 1$ for every $d\notin \{2,5,8\}$, all but $o(n)$ of the vertices have degree $2$, $5$ or $8$. If all degrees lie in $\{2,5,8\}$,
then $n_{2}+n_{5}+n_{8}=n$ and $2n_{2}+5n_{5}+8n_{8}=4n-6$ give $n_{5}=\frac{2n}{3}-2n_{8}-2$, and
$\max \{n_{2},n_{5},n_{8}\}$ is minimised for $n_{8}=(n-12)/9$, which yields the densities
\begin{equation}
\frac{n_{2}}{n}\rightarrow \frac{4}{9},\qquad \frac{n_{5}}{n}\rightarrow \frac{4}{9},\qquad
\frac{n_{8}}{n}\rightarrow \frac{1}{9}.  \label{For_dens}%
\end{equation}

This is precisely the degree sequence for which the authors of \cite{CaroLauriZarb} were looking.

Next we record the elementary fact about ears (vertices of degree two in a MOP) that we need.

\begin{lemma}
\label{Lemma_peel}Let $G$ be a MOP of order $n\geq 5$ and let $S$ be its set of ears. Then no two ears are adjacent,
every vertex is adjacent to at most two ears, $n_{2}\leq n/2$, and $G-S$ is a MOP of order at least $3$. Consequently
$n_{3}+n_{4}\geq 2$.
\end{lemma}

\begin{proof}
If two ears $u,v$ were adjacent, then $N(u)=\{v,w\}$ and $N(v)=\{u,w\}$ for a common neighbour $w$ (every edge of a MOP lies in a
triangle), so $w$ would separate $\{u,v\}$ from the rest of $G$, contradicting the outer Hamiltonian cycle of $G$. An ear
adjacent to $v$ is an end of the fan at $v$ (an inner vertex of the fan has three neighbours), so
$v$ is adjacent to at most two ears; counting the $2n_{2}$ edges at the ears from the other side gives $2n_{2}\leq
2(n-n_{2})$, i.e.\ $n_{2}\leq n/2$. Deleting an ear from a MOP of order at least four leaves a MOP, so $G-S$ is a MOP of
order $n-n_{2}\geq \lceil n/2\rceil \geq 3$. It therefore has at least two vertices of degree $2$; such a vertex $v$ is not an ear
of $G$, so it is adjacent to one or two ears of $G$, and $\deg _{G}(v)\in \{3,4\}$.
\end{proof}

{We can now slightly improve upon Lemma \ref{Lemma_u}, proving the best lower bound, by slightly modifying the weights in the window counting.}

\begin{theorem}
\label{Tm_lower}For every $n\geq 5$,
\[
\mathrm{MOP}(n,2)\ \geq \ \left\lceil \frac{4n+10}{9}\right\rceil ,
\]
and equality holds for every $n\equiv 2\ (\mathrm{mod}\ 18)$ with $n\geq 20$.
\end{theorem}

\begin{proof}

We first show that every MOP $G$ of order $n\geq 5$ satisfies
\[
9\,\mathrm{sp}(G,2)\ \geq \ 4n+6+2n_{3}+3n_{4}+n_{7}+\sum_{d\geq 9}(d-8)n_{d}\ \geq \ 4n+10 .
\] 
{We assign weights which are slightly different to those in Lemma \ref{Lemma_u}, as follows: 
\[
w(d)=6\cdot \mathbf{1}\{2\leq d\leq 4\}+1\cdot \mathbf{1}\{3\leq d\leq 5\}+2\cdot \mathbf{1}\{5\leq d\leq 7\},
\]
that is $w(2)=6$, $w(3)=w(4)=7$, $w(5)=3$, $w(6)=w(7)=2$ and $w(d)=0$ for $d\geq 8$. Summing the three
indicators separately,}
\[
\sum_{v}w(\deg (v))=6W_{[2,4]}+W_{[3,5]}+2W_{[5,7]}\leq (6+1+2)\,\mathrm{sp}(G,2)=9\,\mathrm{sp}%
(G,2).
\]
On the other hand $w(d)\geq 8-d$ for every $d\geq 2$, the defect
$\vartheta (d)=w(d)-(8-d)$ being $0$ for $d\in \{2,5,6,8\}$, and $\vartheta (3)=2$,
$\vartheta (4)=3$, $\vartheta (7)=1$, $\vartheta (d)=d-8$ for $d\geq 9$. Hence
\[
\sum_{v}w(\deg (v))=\sum_{v}\left( 8-\deg (v)\right) +\sum_{v}\vartheta (\deg (v))
\]
\[
=8n-(4n-6)+2n_{3}+3n_{4}+n_{7}+\sum_{d\geq 9}(d-8)n_{d},
\]
which is the first required inequality. Finally $n_{3}+n_{4}\geq 2$ by Lemma \ref{Lemma_peel}, so
$2n_{3}+3n_{4}\geq 4$, which gives the second one. Since $\mathrm{sp}(G,2)$ is an integer, this proves
$\mathrm{MOP}(n,2)\geq \left\lceil (4n+10)/9\right\rceil $ for every $n\geq 5$. The equality for
$n\equiv 2\ (\mathrm{mod}\ 18)$, $n\geq 20$, is witnessed by the graphs $H_{k}$ of Theorem
\ref{Tm_construction} below, which have $\mathrm{sp}(H_{k},2)=(4n+10)/9$ for $n=18k+2$.
\end{proof}

We next supply a construction for $n\equiv 2\ (\mathrm{mod}\ 18)$ realising the lower bound.

\subsection{A construction attaining the bound}
\label{Sec_construction}

We now exhibit maximal outerplanar graphs for which Theorem \ref{Tm_lower} is an equality. The
shape of the construction is a nested ladder: $k$ pairwise nested chords, the \emph{rungs}, cut
the polygon into an outer quadrilateral, $k-1$ blocks and an innermost region; each block is a
triangulated $20$-gon which carries two vertices of degree $8$, eight of degree $5$ and eight ears,
and the innermost region is closed by a fixed cap on $18$ vertices which carries three of the four
unavoidable exceptional vertices (degrees $3,6,6$); the fourth, of degree $3$, lies in the outer quadrilateral.

Throughout this section $k\geq 1$ and $n=18k+2$; we work on the convex polygon with vertices
$0,1,\ldots ,n-1$ in this cyclic order, so that the boundary edges are $\{i,i+1\}$ for
$0\leq i\leq n-2$ together with $\{0,n-1\}$. Put
\[
a_{i}=9i,\qquad b_{i}=18k-1-9i\qquad (0\leq i\leq k-2),\qquad c=9k-9,
\]
and let
\[
R_{i}=\{a_{i},b_{i}\}\ \ (0\leq i\leq k-2),\qquad R_{k-1}=\{c,c+17\}
\]
be the rungs (for $k=1$ there are no $a_{i},b_{i}$ and the only rung is $R_{0}=\{0,17\}$). Notice
that $a_{0}<a_{1}<\cdots <a_{k-2}<c$ and
$b_{0}>b_{1}>\cdots >b_{k-2}>c+17$, the last inequality because $b_{k-2}=c+26$; hence
\[
R_{0}\supset R_{1}\supset \cdots \supset R_{k-1}
\]
in the sense that the arcs they cut off are nested, so the rungs are pairwise non-crossing. They
divide the polygon into the quadrilateral $\Omega _{-1}$ on $\{n-3,n-2,n-1,0\}$ lying outside
$R_{0}$, the $20$-gons
\[
\Omega _{i}=\{a_{i},\ldots ,a_{i}+9\}\cup \{b_{i}-9,\ldots ,b_{i}\}\qquad (0\leq i\leq k-2)
\]
lying between $R_{i}$ and $R_{i+1}$, and the $18$-gon $\Omega ^{\mathrm{in}}$ on
$\{c,c+1,\ldots ,c+17\}$ lying inside $R_{k-1}$.

In each $\Omega _{i}$ we use the local labelling
\begin{equation}
\ell \longmapsto a_{i}+\ell \ \ (0\leq \ell \leq 9),\qquad 10+j\longmapsto b_{i}-9+j\ \ (0\leq
j\leq 9),  \label{For_local}%
\end{equation}

under which the two boundary chords of $\Omega _{i}$ are $\{0,19\}=R_{i}$ and
$\{9,10\}=R_{i+1}$.

\begin{lemma}
\label{Lemma_block}The $17$ pairs
\begin{equation}
\begin{array}
[c]{l}%
\{0,2\},\ \{2,4\},\ \{4,6\},\ \{6,8\};\qquad \{0,11\},\ \{0,13\},\ \{0,15\},\ \{0,17\};\\
\{2,11\},\ \{4,11\},\ \{6,11\},\ \{8,11\};\qquad \{11,13\},\ \{13,15\},\ \{15,17\},\ \{17,19\};\\
\{8,10\}
\end{array}
\label{For_block}%
\end{equation}
are pairwise non-crossing chords of the convex $20$-gon $0,1,\ldots ,19$, so that together with its sides
they triangulate it. The triangulated $20$-gon so obtained, drawn in Figure \ref{Fig_block}, is the
\emph{building block} of the construction of Theorem \ref{Tm_construction}: it is the region $\Omega _{i}$
lying between two consecutive rungs, its sides $\{0,19\}$ and $\{9,10\}$ being the rungs $R_{i}$ and
$R_{i+1}$ themselves, along which consecutive blocks are glued.
\end{lemma}

\begin{proof}
A convex polygon with $N$ sides is triangulated by any $N-3$ pairwise non-crossing chords, and
$17=20-3$, so only the non-crossing has to be checked. Two chords $\{a,b\}$ and $\{c,d\}$ of the
polygon $0,1,\ldots ,19$, written with $a<b$ and $c<d$, cross if and only if exactly one of $c,d$ lies
strictly between $a$ and $b$. With this criterion the four groups of (\ref{For_block}) are immediate: the
chords $\{0,2\},\{2,4\},\{4,6\},\{6,8\},\{8,10\}$ cut off the pairwise disjoint arcs $1,3,5,7,9$; the chords
$\{11,13\},\{13,15\},\{15,17\},\{17,19\}$ cut off the pairwise disjoint arcs $12,14,16,18$; the fan
$\{0,11\},\{0,13\},\{0,15\},\{0,17\}$ shares the endpoint $0$ and its chords are nested; and the fan
$\{2,11\},\{4,11\},\{6,11\},\{8,11\}$ shares the endpoint $11$ and is nested likewise. Between the groups
there is nothing to check either, because the arcs $\{1,\ldots ,10\}$ and $\{12,\ldots ,19\}$ cut off by the
two fans are disjoint and each of the remaining chords lies inside one of them. All of this is visible in
Figure \ref{Fig_block}.
\end{proof}

\begin{figure}[ht]
\begin{center}
\includegraphics[width=0.40\textwidth]{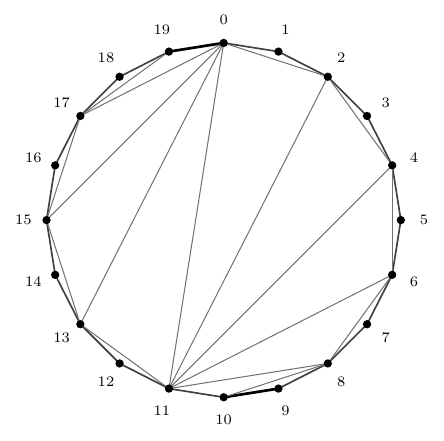}
\end{center}
\caption{The building block: the convex $20$-gon with the $17$ chords (\ref{For_block}). Its two sides
$\{0,19\}$ and $\{9,10\}$, drawn in bold, are the rungs $R_{i}$ and $R_{i+1}$ of Theorem
\ref{Tm_construction}; in the global labelling of that theorem the vertices $0,\ldots ,9$ are
$a_{i},\ldots ,a_{i}+9$ and the vertices $10,\ldots ,19$ are $b_{i}-9,\ldots ,b_{i}$. The two vertices of
degree $8$ of the block are $0$ and $11$.}%
\label{Fig_block}%
\end{figure}

\begin{lemma}
\label{Lemma_cap}Let $\Gamma $ consist of the $15$ chords $\{c+x,c+y\}$ with
\begin{equation}
(x,y)\in \left\{
\begin{array}
[c]{l}%
(0,13),(0,15),(1,3),(1,12),(1,13),(3,5),(3,10),(3,12),\\
(5,7),(5,10),(7,9),(7,10),(10,12),(13,15),(15,17)
\end{array}
\right\} .  \label{For_cap}%
\end{equation}

Then $\Gamma $ triangulates the $18$-gon $\Omega ^{\mathrm{in}}$.
\end{lemma}

\begin{proof}
We argue in the local labels $0,1,\ldots ,17$; recall that $\{0,17\}$ is the side $R_{k-1}$ of
$\Omega ^{\mathrm{in}}$ and not a chord, and that none of the $15$ listed pairs is a side. That no two of them
cross is visible in Figure \ref{Fig_cap}, where they are drawn, and is confirmed by the criterion used in the
proof of Lemma \ref{Lemma_block}. Finally $15=18-3$, and a family of that many pairwise non-crossing chords of a
convex $18$-gon is a triangulation.
\end{proof}

\begin{figure}[ht]
\begin{center}
\includegraphics[width=0.40\textwidth]{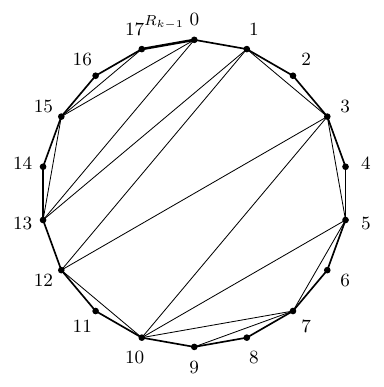}\hfill
\includegraphics[width=0.52\textwidth]{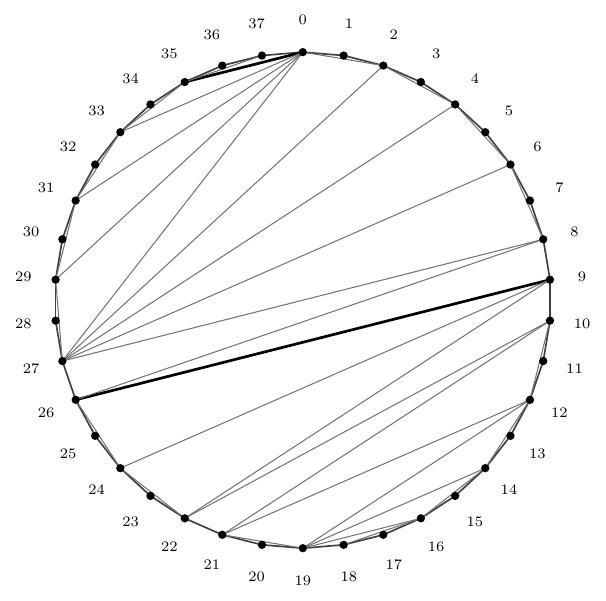}
\end{center}
\caption{Left: the cap $\Gamma $, the $18$-gon $\Omega ^{\mathrm{in}}$ in its local labels $0,\ldots ,17$ with
the $15$ chords (\ref{For_cap}); the side $\{0,17\}$ (bold) is the rung $R_{k-1}$. Right: the whole ladder
$H_{2}$ on $n=18\cdot 2+2=38$ vertices, with its two rungs $R_{0}=\{0,35\}$ and $R_{1}=\{9,26\}$ in bold.
They cut the polygon into the quadrilateral on $\{35,36,37,0\}$, the single block $\Omega _{0}$ on
$\{0,\ldots ,9\}\cup \{26,\ldots ,35\}$ and the cap $\Omega ^{\mathrm{in}}$ on $\{9,\ldots ,26\}$. Its degree
sequence is $2^{(16)}3^{(2)}5^{(16)}6^{(2)}8^{(2)}$ and $\mathrm{sp}(H_{2},2)=18=\lceil (4\cdot 38+10)/9\rceil $.}%
\label{Fig_cap}%
\end{figure}

\begin{theorem}
\label{Tm_construction}Let $k\geq 1$, $n=18k+2$, and let $H_{k}$ be the graph on
$0,1,\ldots ,n-1$ consisting of all $n$ boundary edges of the polygon together with the
$18k-1$ chords
\[
\mathcal{C}(H_{k})=\{n-3,n-1\}\ \cup \ \bigl\{ R_{i}:0\leq i\leq k-1\bigr\} \ \cup \
\bigcup_{i=0}^{k-2}\mathcal{B}_{i}\ \cup \ \Gamma ,
\]
where $\mathcal{B}_{i}$ is the set of $17$ chords of $\Omega _{i}$ given by (\ref{For_block}) in
the local labelling (\ref{For_local}), and $\Gamma $ is given by (\ref{For_cap}). Then
$H_{k}$ is a maximal outerplanar graph with degree sequence
\[
2^{(8k)}\ 3^{(2)}\ 5^{(8k)}\ 6^{(2)}\ 8^{(2k-2)},
\]
and
\[
\mathrm{sp}(H_{k},2)=8k+2=\frac{4n+10}{9}.
\]

\end{theorem}

\begin{proof}
The listed pairs are pairwise distinct: two of them lying in different regions could coincide only
if both were the rung shared by these regions, and no pair of (\ref{For_block}) is $\{0,19\}$ or
$\{9,10\}$, no pair of (\ref{For_cap}) is $\{0,17\}$. Hence the number of chords is
$1+k+17(k-1)+15=18k-1=n-3$. None of them is a boundary edge: the gap between the two endpoints of
each listed pair is at least $2$, and the side $\{0,n-1\}$ does not occur, because $n-1$ belongs
to $\Omega _{-1}$ only, whose single chord is $\{n-3,n-1\}$.

Notice next that the rungs are pairwise non-crossing, as observed above, every chord
of $\mathcal{B}_{i}$ has both endpoints in $\Omega _{i}$, every chord of $\Gamma $ has both
endpoints in $\Omega ^{\mathrm{in}}$, and $\{n-3,n-1\}$ has both endpoints in $\Omega _{-1}$.
Two chords lying in two different regions of a non-crossing family cannot cross, inside
$\Omega _{i}$ the assertion is Lemma \ref{Lemma_block}, inside $\Omega ^{\mathrm{in}}$ it is Lemma
\ref{Lemma_cap}, and $\Omega _{-1}$ carries a single chord. Hence $\mathcal{C}(H_{k})$ consists of
$n-3$ pairwise non-crossing chords, so it triangulates the polygon and $H_{k}$ is a MOP.

We read off the degrees. Fix $i\leq k-2$ and write $a=a_{i}$, $b=b_{i}$. The vertex $a$ has the two
boundary neighbours $a\pm 1$ and the six chords $\{a,a+2\}$, $R_{i}=\{a,b\}$ and
$\{a,b-2t\}$ for $t=1,2,3,4$, so $\deg (a)=8$; symmetrically $b-8$ has the boundary neighbours
$b-9,b-7$ and the six chords $\{a+2t,b-8\}$ for $t=0,1,2,3,4$ and $\{b-8,b-6\}$, so
$\deg (b-8)=8$. For $t=1,2,3$ the vertex $a+2t$ has the boundary neighbours $a+2t\pm 1$ and the
three chords $\{a+2t-2,a+2t\}$, $\{a+2t,a+2t+2\}$, $\{a+2t,b-8\}$, so $\deg (a+2t)=5$; the vertex
$a+8$ has $a+7,a+9$ and the chords $\{a+6,a+8\}$, $\{a+8,b-8\}$, $\{a+8,b-9\}$, so
$\deg (a+8)=5$; and for $t=1,2,3$ the vertex $b-2t$ has $b-2t\pm 1$ and
$\{b-2t-2,b-2t\}$, $\{b-2t,b-2t+2\}$, $\{a,b-2t\}$, so $\deg (b-2t)=5$. The vertex $b_{i}$ has, for
$i\geq 1$, the boundary neighbours $b_{i}\pm 1$ and the chords $\{a_{i-1}+8,b_{i-1}-9\}$,
$R_{i}$ and $\{b_{i}-2,b_{i}\}$, so $\deg (b_{i})=5$, and for $i=0$ the vertex $b_{0}=n-3$ has
$n-4,n-2$ and the chords $R_{0}$, $\{n-3,n-1\}$, $\{b_{0}-2,b_{0}\}$, so again $\deg (b_{0})=5$.
The eight vertices $a+2t-1$ and $b-2t+1$ with $1\leq t\leq 4$ have no incident chord and are ears.
Thus each block contributes eight ears, eight vertices of degree $5$ and two of degree $8$.

On the inner cap, $c$ has the boundary neighbours $c\pm 1$ and the chords
$\{c,c+13\},\{c,c+15\},R_{k-1}$, so $\deg (c)=5$, and $c+17$ has $c+16,c+18$ and the chords
$\{a_{k-2}+8,b_{k-2}-9\}$ (for $k\geq 2$; for $k=1$ this role is played by the outer chord
$\{n-3,n-1\}=\{17,19\}$), $R_{k-1}$, $\{c+15,c+17\}$, so $\deg (c+17)=5$. Reading off
(\ref{For_cap}) in the same way one gets
\[
\bigl( \deg (c),\deg (c+1),\ldots ,\deg (c+17)\bigr) =(5,5,2,6,2,5,2,5,2,3,6,2,5,5,2,5,2,5),
\]
so the cap contributes seven ears, eight vertices of degree $5$, two of degree $6$ and one of
degree $3$; and $\deg (n-2)=2$, $\deg (n-1)=3$. Since the sets $\{a_{i},\ldots ,a_{i}+8\}$,
$\{b_{i}-8,\ldots ,b_{i}\}$ $(0\leq i\leq k-2)$, $\{c,\ldots ,c+17\}$ and $\{n-2,n-1\}$ partition
$\{0,\ldots ,n-1\}$, we obtain
\[
n_{2}=8(k-1)+7+1=8k,\quad n_{3}=2,\quad n_{5}=8(k-1)+8=8k,\quad n_{6}=2,\quad n_{8}=2(k-1),
\]
which indeed sum to $18k+2=n$ and give
$2\cdot 8k+3\cdot 2+5\cdot 8k+6\cdot 2+8(2k-2)=72k+2=4n-6$.

Finally the windows are
\[
W_{[2,4]}=n_{2}+n_{3}=8k+2,\quad W_{[3,5]}=n_{3}+n_{5}=8k+2,\quad
W_{[4,6]}=W_{[5,7]}=n_{5}+n_{6}=8k+2,
\]
\[
W_{[6,8]}=n_{6}+n_{8}=2k,\qquad W_{[7,9]}=W_{[8,10]}=2k-2,
\]
and all further windows are empty, so $\mathrm{sp}(H_{k},2)=8k+2$; and
$9(8k+2)=72k+18=4(18k+2)+10=4n+10$.
\end{proof}

Notice the perfect balance $W_{[2,4]}=W_{[3,5]}=W_{[4,6]}=W_{[5,7]}$: this is what makes $H_{k}$
optimal. In the notation of Theorem \ref{Tm_lower} we have $n_{4}=n_{7}=0$ and $n_{3}=2$, so the
middle term there equals $4n+6+4=4n+10$ exactly.

For the other residue classes of $n$ modulo $18$ we have found seventeen further caps with which the same ladder
attains the bound as well; we have verified by computer that the resulting graphs are maximal outerplanar and that
their window number equals $\left\lceil (4n+10)/9\right\rceil $ for every order in the range covered, so that
presumably $\mathrm{MOP}(n,2)=\left\lceil (4n+10)/9\right\rceil $ for every $n\geq 14$. The caps and the
verification are in the supplementary material; here we do not use them, and prove instead that the bound is attained
up to an additive constant for every $n$.

{We are now ready to show that $4n/9$ is the right asymptotic bound for $\mathrm{MOP}(n,2)$.}

\begin{theorem}
\label{Tm_main}For every $n\geq 5$,
\[
\mathrm{MOP}(n,2)\ \geq \ \left\lceil \frac{4n+10}{9}\right\rceil ,
\]
with equality when $n\equiv 2\ (\mathrm{mod}\ 18)$, and $\mathrm{MOP}(n,2)\leq \left\lceil (4n+10)/9\right\rceil +43$ for
every $n\geq 20$. In particular $\mathrm{MOP}(n,2)=4n/9+O(1)$.
\end{theorem}

\begin{proof}
The lower bound is Theorem \ref{Tm_lower} (which holds for every $n\geq 5$), and the equality for $n=18k+2$ is Theorem \ref{Tm_construction}. For the
upper bound let $n\geq 20$, $k=\left\lfloor (n-2)/18\right\rfloor \geq 1$ and $r=n-18k-2\in \{0,\ldots ,17\}$, and let
$G$ be obtained from $H_{k}$ by attaching an ear to each of $r$ distinct outer edges; $G$ is a MOP of order $n$.
Attaching one ear adds a vertex of degree $2$ and raises the degrees of two vertices by one, so every window count
$W_{[p,p+2]}$ increases by at most $3$. Since $\left\lceil (4n+10)/9\right\rceil =8k+2+\left\lceil 4r/9\right\rceil $,
we get
\[
\mathrm{sp}(G,2)\leq \mathrm{sp}(H_{k},2)+3r=\left\lceil \frac{4n+10}{9}\right\rceil +3r-\left\lceil \frac{4r}{9}\right\rceil
\leq \left\lceil \frac{4n+10}{9}\right\rceil +43 .
\]
\end{proof}

\section{Concluding remarks}
{We have developed a weighted counting argument based on the notion of windows, which optimizes a not necessarily graphical version of $\mathrm{sp}(G,k)$  via $\delta$-sequences.   This is done in Theorem \ref{Tm_gen}.  We then showed via Proposition \ref{Prop_equality} and Theorem \ref{Tm_sequence}, that optimal $\delta$-sequences force the elements of the sequence to be clusters in gaps forming arithmetic progressions.  Lastly, we show, in Theorem \ref{Tm_exact}, that for $n$ large enough and  fixed values of $\delta$, $k$,  and $d$, the $\delta$-sequences identified in Theorem \ref{Tm_sequence} are graphical and hence optimal, and we gave a construction in Proposition \ref{Prop_construction}.}

{For maximal outerplanar graphs we completed the asymptotic bound of $\mathrm{MOP}(n,k )$ which appeared in \cite{CaroLauriZarb} by proving the missing case $\mathrm{MOP}(n,2) \sim 4n/9$  in Theorem  \ref{Tm_gen} of \cite{CaroLauriZarb}. }

{The same counting argument applied to maximal planar graphs is the subject of the companion paper \cite{P2}, where we asymptotically solved the questions raised in Problem 2 of \cite{CaroLauriZarb}.}

We conclude with a problem concerning the optimal value of $n_{0}(\delta ,k,d)$.

\begin{problem}
\label{Pr_threshold}Determine the smallest $n_{0}(\delta ,k,d)$ from which the maximised bound of
Theorem \ref{Tm_gen} is exact for all graphs. The data suggest $n_{0}=2d+O(h)$ rather than the
quadratic bound of Theorem \ref{Tm_exact}.
\end{problem}
\bigskip

\bigskip\noindent\textbf{Disclosure on AI.}~~{In preparing this paper the authors used large language models of
Anthropic (Claude Opus~5) for computer experiments and exploration concerning $\mathrm{MOP}(n,2 )$, as well as helping to implement the window counting, checking computations and proofs, and for drafting and editing parts of the text. The authors take full responsibility for its content.}


\end{document}